\documentclass[a4paper,11pt,oneside]{article} 
\usepackage{amsfonts,amssymb}
\usepackage{color} \usepackage{mathrsfs} \let\mathcal\mathscr
\usepackage[all,ps,cmtip]{xy}
\usepackage{lscape}\usepackage{amsthm}

\usepackage[utf8]{inputenc}
\usepackage{fancyhdr}
\usepackage[T1]{fontenc}
\usepackage{graphicx}
\usepackage{tikz}
\usepackage[left=3cm,right=3cm,top=4cm,bottom=4cm]{geometry}
\usepackage{amsmath,amsfonts,amssymb}
\usepackage{eurosym}
\usepackage{skull}
\usepackage{mathenv}
\usepackage{pstricks}

\title{\sc On quartic double fivefolds and the matrix factorizations of exceptional quaternionic representations}
\author{\sc Roland Abuaf \footnote{E-mail :\textit{rabuaf@gmail.com}.}}

\usepackage{amsfonts,amssymb}
\usepackage{color} \usepackage{mathrsfs} \let\mathcal\mathscr

\usepackage[all,ps,cmtip]{xy}

\DeclareGraphicsExtensions{eps}
\DeclareGraphicsExtensions{pdf}

\usepackage{amsmath}
\usepackage{latexsym}

\newtheorem{theo}{Theorem}[subsection]
\newtheorem{theo*}{Theorem}

\newtheorem{rem}[theo]{Remark}
\newtheorem{prop}[theo]{Proposition}
\newtheorem{prop*}{Proposition}

\newtheorem{prob*}{Problem}
\newtheorem{quest}[theo]{Question}

\newtheorem{defi}[theo]{Definition}

\newtheorem{lem}[theo]{Lemma}
\newtheorem{cor}[theo]{Corollary}

\def\DB{\mathrm{D^{b}}}
\def\OO{\mathcal{O}}

\def\R0{\mathrm{R^{0}}}

\def\HHH{\mathrm{Hom}}

\def\OO{\mathcal{O}}

\def\ad{\mathrm{ad}}

\def\C{\mathcal{C}}
\def\sQ{\mathcal{Q}_{\mathrm{Spin}_{12}}}
\def\sP{P_{\mathrm{Spin}_{12}}}
\def\lQ{\mathcal{Q}_{\mathrm{SL}_{6}}}
\def\lP{P_{\mathrm{SL}_{6}}}
\def\F{\mathcal{F}}
\def\E{\mathcal{E}}
\def\GG{\mathrm{G}}
\def\G{\mathcal{G}}

\def\Q{\mathcal{Q}}
\def\T{\mathcal{T}}
\def\A{\mathcal{A}}
\def\Sp{\mathrm{Spin}}

\def\X{\mathcal{X}}

\DeclareMathOperator{\Pff}{Pff}

\newcommand{\eq}[1][r]
{\ar@<-3pt>@{-}[#1]
\ar@<-1pt>@{}[#1]|<{}="gauche"
\ar@<+0pt>@{}[#1]|-{}="milieu"
\ar@<+1pt>@{}[#1]|>{}="droite"
\ar@/^2pt/@{-}"gauche";"milieu"
\ar@/_2pt/@{-}"milieu";"droite"}

\newcommand{\incl}[1][r]
  {\ar@<-0.2pc>@{^(-}[#1] \ar@<+0.2pc>@{-}[#1]}

{

\begin{document}

\maketitle

\begin{abstract}
We study quartic double fivefolds from the perspective of Fano manifolds of Calabi-Yau type and that of exceptional quaternionic representations. We first prove that the generic quartic double fivefold can be represented, in a finite number of ways, as a double cover of $\mathbb{P}^5$ ramified along a linear section of the $\mathrm{Spin}_{12}$-invariant quartic in $\mathbb{P}^{31}$. Then, using the geometry of the Vinberg's type $\mathrm{II}$ decomposition of some exceptional quaternionic representations, and backed by some cohomological computations performed by Macaulay$2$, we prove the existence of a spherical rank $6$ vector bundle on such a generic quartic double fivefold. We finally use the existence this vector bundle to prove that the homological unit of the CY-$3$ category associated by Kuznetsov to the derived category of a generic quartic double fivefold is $\mathbb{C} \oplus \mathbb{C}[3]$.
\end{abstract}

\vspace{\stretch{1}}

\newpage

\tableofcontents

\begin{section}{Introduction}
\subsection{Manifolds of Calabi-Yau type}

Manifolds of Calabi-Yau type were defined by Iliev and Manivel \cite{maniliev2} as compact complex manifolds of odd dimension whose middle dimensional Hodge structure is similar to that of a Calabi-Yau threefold. More precisely, following \cite{maniliev2}:

\begin{defi}\label{defmcyt}
Let $X$ be a smooth complex compact algebraic variety of odd dimension $2n+1$, for $n \geq 1$. We say that $X$ is of Calabi-Yau type if the following three conditions hold:
\begin{enumerate}
\item The middle dimensional Hodge structure of $X$ is numerically similar to that of Calabi-Yau threefold, that is:
\begin{equation*}
\begin{split}
h^{n+2,,n-1}(X) = 1, & \hspace{0.5cm} \textrm{and} \hspace{0.5cm} h^{n+p+1,n-p}(X) = 0, \,\, \textrm{for} \,\, p \geq 2.
\end{split}
\end{equation*}
 \item For any generator $\tau \in H^{n+2,n-1}(X)$, the contraction map:
 \begin{equation*}
 H^1(X,T_X) \stackrel{\cap \, \tau}\longrightarrow H^{n-1}(X,\Omega^{n+1}(X)
 \end{equation*}
 is an isomorphism.
 
\item The Hodge numbers $h^{k,0}(X)$ vanish, for all $1 \leq k \leq 2n$.
\end{enumerate}

\begin{rem} \label{remfonda}
By Serre duality, a smooth threefold with trivial canonical bundle automatically satisfies condition $2$ in the above definition. On the other hand, for smooth manifolds of dimension bigger than four, it seems highly non-trivial to check this condition.  
\end{rem}

\end{defi}

Potential examples of Fano manifolds of Calabi-Yau type (namely the cubic sevenfold and the quartic double fivefold) appeared some time ago in the Physics literature (see \cite{CDP, Sch, BBVW}). They were used to describe the mirrors of some rigid Calabi-Yau threefolds obtained as crepant resolutions of product of elliptic curves divided by some well-chosen finite groups (we refer to \cite{CDP, Sch, BBVW} for more details). These examples have been put into a more systematic mathematical treatment in \cite{maniliev2}. Physicists were however not too far from exhausting all possible examples of complete intersections in (weighted) projective spaces which should be of Calabi-Yau type. Indeed, an inspection of Hodge numbers for smooth complete intersections in weighted projective spaces reveals the following:

\begin{prop}[section 3.1 of \cite{maniliev2}] \label{prop1} Let $X$ be a smooth complete intersection of Calabi-Yau type in a weighted projective space. Assume that $\dim X \geq 4$, then $X$ is necessarily one of the following:

\begin{enumerate}
\item a smooth cubic sevenfold in $\mathbb{P}^8$,
\item a smooth quartic double fivefold in $\mathbb{P}(1,1,1,1,1,1,2)$,
\item a transverse intersection of a smooth cubic and a smooth quadric in $\mathbb{P}^7$.
\end{enumerate}
\end{prop}

\begin{rem}
We do not assert that all examples appearing in the above proposition are of Calabi-Yau type. Indeed, as mentioned in remark \ref{remfonda}, it is non-trivial to check that condition $2$ in definition \ref{defmcyt} holds for these examples. As far as generic cubic sevenfolds and generic quartic double fivefolds are concerned, Iliev and Manivel proved that they are of Calabi-Yau type if and only if they can be represented, in a finite number of ways, as a linear sections of the projective dual of a certain homogeneous space associated to them (see \cite{maniliev2}, section $4.4$, proposition $4.5$ and remarks thereafter). For the cubic sevenfold, this finite representation property is proved in \cite{maniliev}. The quartic double fivefold will be dealt with in the present paper.

The case of the transverse intersection of a smooth cubic and a smooth quadric in $\mathbb{P}^7$ is a bit more mysterious. Indeed, it is not immediately clear what finite representation statement should be equivalent to the fact that it is of Calabi-Yau type. That would certainly be a very interesting example to explore in more details.
\end{rem}

The manifolds exhibited in proposition \ref{prop1} have a lot in common with the archetypal Calabi-Yau threefold : the quintic threefold. We will enumerate some interesting geometric and cohomological properties of the quintic threefold which have been (even partially) shown to be true for the above manifolds. 

\begin{enumerate}

\item Clemens proved that the Griffiths group of smooth quintic in $\mathbb{P}^4$ is not finitely generated \cite{Clemens}. Favero, Iliev and Katzarkov defined a notion of Griffiths group for the manifolds which appear in proposition \ref{prop1}. Using some earlier work of Voisin \cite{voisin}, they showed that the Griffiths group of these manifolds have an infinitely generated Griffiths group \cite{FIK}. We refer to the earlier work of Albano and Collino for the case of the cubic sevenfold \cite{Collino}.

\item Beauville proved that a generic quintic in $\mathbb{P}^4$ has a finite number of determinantal representations \cite{beauville}. In \cite{maniliev}, Iliev and Manivel generalized Beauville's result in the case of cubic sevenfold : a generic cubic in $\mathbb{P}^8$ can be represented, in a finite number of ways, as a linear sections of the $E_6$-invariant cubic hypersurface in $\mathbb{P}^{26}$, the Cartan cubic.

\item Any line bundle on the quintic threefold is \textit{spherical} (i.e. its Ext-algebra is isomorphic to $\mathbb{C} \oplus \mathbb{C}[3]$) and spherical objects provide non-trivial auto-equivalences of the derived category of the quintic threefold (see \cite{ST}). Kuznetsov proved that the derived categories of the manifolds appearing in proposition \ref{prop1} always contain a semi-orthogonal component which is a Calabi-Yau category of dimension $3$ \cite{kuz1}. Furthermore, Iliev and Manivel exhibited examples of spherical vector bundles contained in the CY-$3$ category associated to the derived category of a generic cubic sevenfold \cite{maniliev}.

\item Generic quintinc threefolds are endowed with a so-called Yukawa coupling which satisfies very interesting equations (see \cite{morrison} for instance). It is explained in \cite{CDP} that similar properties hold for the Yukawa coupling constructed on generic cubic sevenfolds. \cite{maniliev2} ask if this could be true for any manifold of Calabi-Yau type.

\end{enumerate}
Obviously, we do not claim that this enumeration is exhaustive in any sense. In fact, this is quite the opposite : we hope that many other remarkable features of the quintic threefold will be shared by the complete intersection manifolds of Calabi-Yau type.

\subsection{Generic quartic double fivefolds}

In this paper, we will focus on quartic double fivefolds. Any such manifold is the zero locus of a weighted homogeneous polynomial of the form $f_4(z_1,\ldots,z_6) + x^2$ in $\mathbb{P}(1,1,1,1,1,1,2)$ where $f_4$ is an element of $S^4 \mathbb{C}^6$. Our main results are the following:

\begin{theo*}[see Theorem \ref{section}] \label{theo1}
The generic quartic double fivefold can be represented, in a finite number of ways, as a double cover of $\mathbb{P}^5$ ramified along a linear section of the $\mathrm{Spin}_{12}$-invariant quartic $\sQ \subset \mathbb{P}^{31}$ (the Igusa quartic). As a consequence, the generic quartic double fivefold is a manifold of Calabi-Yau type.
\end{theo*}

\begin{theo*}[see Theorem \ref{mainmain}] \label{theo2}
The $3$ dimensional Calabi-Yau category associated to the derived category of the generic quartic double fivefold contains a rank $6$ spherical vector bundle.
\end{theo*}

Our proof of Theorem \ref{theo1} uses the strategy already highlighted in \cite{adler, beauville, maniliev}. Namely, if $\sP$ is an equation for the $\mathrm{Spin}_{12}$ invariant quartic $\sQ \subset \mathbb{P}^{31}$, we prove that the pull-backs of the partial derivatives of $\sP$ to $\mathbb{C}[z_1,\ldots,z_6]$ by a generic $32 \times 6$ matrix generate $S^4 \mathbb{C}^6$. We then deduce that the natural map:

\begin{equation*}
\GG(6,\Delta) /\!\!/ \mathrm{Spin}_{12} \longrightarrow S^4 \mathbb{C}^6 /\!\!/  \mathrm{GL}_6
\end{equation*}
which associates to $L \in \GG(6,\Delta)$ its intersection with $\sQ \subset \mathbb{P}(\Delta)$ is generically \'etale. The computation of the dimension of the space generated by the pull-backs of the partial derivatives of $\sP$ to $\mathbb{C}[z_1,\ldots,z_6]$ is done using Macaulay$2$ \cite{M2}.

\bigskip

In order to demonstrate Theorem \ref{theo2}, we use the basic geometry of some exceptional quaternionic representations (\cite{GW1, clerc, slupistan}). We first start with the Vinberg type $\mathrm{II}$ decomposition of the Lie algebra $\frak{e_6}$:
\begin{equation} \label{equa0}
\frak{e_6} = \mathbb{C}^* \oplus (\bigwedge^3 \mathbb{C}^6)^* \oplus \frak{gl}_6 \oplus \bigwedge^3 \mathbb{C}^6 \oplus \mathbb{C},
\end{equation}
The properties of type $\mathrm{II}$ grading for exceptional quaternionic representations representations entail that for any $y \in \bigwedge^3 \mathbb{C}^6$, we have:

\begin{equation*}
(\ad_y^{\mathfrak{e_6}})^4(X_{-\beta}) = \lP(y).X_{\beta},
\end{equation*}
where $X_{-\beta}$ and $X_{\beta}$ are generators of the one-dimensional factors appearing in degree $-2$ and $2$ in the decomposition (\ref{equa0}) and $\lP$ is the equation of the $\mathrm{SL}_6$ invariant quartic $\lQ \subset \mathbb{P}(\bigwedge^3 \mathbb{C}^6)$. As $(\ad^{\mathfrak{e_6}}_y)^2$ is an element of $\frak{gl}_6^* \simeq \frak{gl}_6$, we deduce that the pair $((\ad^{\mathfrak{e_6}}_y)^2,(\ad^{\mathfrak{e_6}}_y)^2)$ is a matrix factorization of $\lP$ (see lemma \ref{matfacto}). The matrices $B = (\ad^{\mathfrak{e_6}}_y)^2 + ix.I_6$ and $C = (\ad^{\mathfrak{e_6}}_y)^2 -ix.I_6$ form therefore a matrix factorization of $\lP(y) + x^2$. Restricting $B$ and $C$ to a generic $\mathbb{P}^5 \subset \mathbb{P}(\bigwedge^3 \mathbb{C}^6)$, we deduce the existence of a specific matrix factorization for the quartic double fivefold ramified along a generic fourfold linear section of $\lQ$. Cohomological properties of the restriction of the cokernel of $B$ to the quartic double fivefold determined by the choice of this $\mathbb{P}^5$ will play a crucial role in the proof of Theorem \ref{theo2}.

\bigskip

In \cite{maniliev}, Iliev and Manivel used similar ideas in order to construct a spherical rank $9$ vector bundle on the generic cubic sevenfold. Their proof that this vector bundle is spherical highlights the impressive virtuosity of the authors in manipulating the Borel-Bott-Weil Theorem in type $E_6$. Our cohomological study of the cokernel of $B$ is certainly less elegant but it has the merit to be more accessible to the layman : we compute the necessary Ext-groups using Macaulay$2$ \cite{M2}. The implementation of the matrix representing $(\ad_y^{\mathfrak{e_6}})^2 +ix.I_6$ follows an explicit description given in \cite{kimura}. 

\bigskip

In order to prove Theorem \ref{theo2}, we consider the type $\mathrm{II}$ decomposition:

\begin{equation}\label{equa00} 
\frak{e_7} = \mathbb{C}^* \oplus \Delta^* \oplus \frak{so}_{12} \oplus \mathbb{C} \oplus \Delta \oplus \mathbb{C}.
\end{equation}
Once again, the properties of type $\mathrm{II}$ grading for quaternionic representations yield that for any $z \in \Delta$, we have:

\begin{equation*}
(\ad_z^{\mathfrak{e_7}})^4(X_{-\beta}) = \sP(z).X_{\beta},
\end{equation*}
where $X_{-\beta}$ and $X_{\beta}$ are generators of the one-dimensional factors appearing in degree $-2$ and $2$ in the decomposition (\ref{equa00}) and $\sP$ is the equation of the $\mathrm{Spin}_{12}$-invariant quartic $\sQ \subset \mathbb{P}(\Delta)$. 

\bigskip

Let $L \subset \mathbb{P}(\Delta)$ be a generic $\mathbb{P}^5$ and denote by $X_L$ the quartic double fivefold ramified over $L \cap \sQ$. The restrictions to $L$ of the matrices $\tilde{B} = (\ad^{\mathfrak{e_7}}_z)^2 + ix.I_{12}$ and $\tilde{C} = (\ad^{\mathfrak{e_7}}_z)^2 -ix.I_{12}$ provide a matrix factorization of the equation of $X_L$ in $\mathbb{P}(1,\ldots,1,2)$. If $L_0$ is a generic $\mathbb{P}^5$ in $\mathbb{P}(\bigwedge^3 \mathbb{C}^6)$, we are able to relate the cohomological properties of the restriction to $X_L$ of cokernel of $\tilde{B} (\ad^{\mathfrak{e_7}}_z)^2 + ix.I_{12}$ to those of the restriction to $X_{L_0}$ of the cokernel of $B(\ad^{\mathfrak{e_6}}_y)^2 + ix.I_6$. We then deduce that a twist of the restriction to $X_L$ of the cokernel of $\tilde{B}$ is the spherical rank $6$ vector bundle whose existence is claimed in Theorem \ref{theo2}.

\bigskip

In the last section of this paper, we discuss a <<topological>> application of the existence of a spherical vector bundle on the generic quartic double fivefold. In \cite{homounit}, the concept of \textit{homological unit} was introduced for a large class of triangulated categories as a replacement for the algebra $H^{\bullet}(\OO_X)$ when the category under study is not (necessarily) the derived category of a projective variety. We prove here (see section $3.3$) that the homological unit of the $3$ dimensional Calabi-Yau category associated to the derived category of a generic quartic double fivefold is $\mathbb{C} \oplus \mathbb{C}[3]$. This computation shows that the $3$-dimensional Calabi-Yau category associated to the derived category of a generic double quartic fivefold is really a non-commutative analogue of a Calabi-Yau threefold, and not just a triangulated category whose Serre functor is the shift by $[3]$.

\vspace{0.5cm}

\noindent \textbf{Acknowledgements :} I am very grateful to Matt B. Young for patiently explaining to me some aspects of \cite{CDP}, \cite{Sch} and \cite{BBVW}. I am very thankful to Laurent Manivel for crucial pointers in the proofs of Lemmas \ref{matfacto} and \ref{matfacto2} and for pointing out to me that it is non-trivial to check condition $2$ in the definition of manifolds of Calabi-Yau type.

\end{section}

\begin{section}{Linear sections of the $\mathrm{Spin}_{12}$ invariant quartic in $\mathbb{P}^{31}$}

\subsection{The $\mathrm{Spin}_{12}$ quartic invariant in $\mathbb{P}^{31}$}

We will start this section with a quick reminder of the work of Igusa \cite{igusa} on spinors of dimension $12$. Let $\Delta$ be one of the half-spin representation of dimension $32$ for the group $\Sp_{12}$. Recall that, as a vector space, we have:

\begin{equation*}
\Delta \simeq \mathbb{C} \oplus \bigwedge^2 \mathbb{C}^6 \oplus \bigwedge^4 \mathbb{C}^6 \oplus \bigwedge^6 \mathbb{C}^6
\end{equation*}
Igusa proves that $\mathbb{P}(\Delta)$ is a prehomogeneous space for $\Sp_{12}$ and that this space has a relative invariant which is of degree $4$. Let us denote by $P_{\mathrm{Spin}_{12}}$ the equation of this relative invariant and by $\Q_{\mathrm{Spin}_{12}} \subset \mathbb{P}^{31}$ the corresponding quartic hypersurface. Igusa gives an explicit description for $P_{\mathrm{Spin}_{12}}$, namely :

\begin{equation*}
 P_{\mathrm{Spin}_{12}}(x) = x_0 \Pff((x_{i,j})) + y_0 \Pff((y_{i,j})) + \sum_{i<j} \Pff((X_{i,j})) \Pff((Y_{i,j})) - \dfrac{1}{4} \left(x_0y_0-\sum_{i<j} x_{i,j}y_{i,j} \right)^2,
\end{equation*}

for $x = x_0 + \sum_{i<j} x_{i,j}.e_i \wedge e_j + \sum_{i<j} y_{i,j}.(e_i \wedge e_j)^* + y_0.e_1 \wedge e_2 \wedge e_3 \wedge e_4 \wedge e_5 \wedge e_6$ in which $(x_{i,j})$ (resp. $(y_{i,j})$) is the alternating matrix determined by the $x_{i,j}$ (resp. $y_{i,j}$) and $(X_{i,j})$ (resp. $(Y_{i,j})$) is the alternating matrix obtained from $(x_{i,j})$ (resp. $(y_{i,j})$) by crossing out its $i$-th and $j$-th lines and columns.

\begin{rem}
The quartic hypersurface in $\mathbb{P}^{31}$ which equation is given above is the tangent variety of the spinor variety $\mathbb{S}_{12} \subset \mathbb{P}^{31}$. This spinor variety appears in the third line of the Tits-Freudenthal magic square as the symplectic Grassmannian $\GG_{\omega}(\mathbb{H}^3,\mathbb{H}^6)$ where $\mathbb{H}$ is the algebra of complexified quaternions. In \cite{manivel}, Landsberg and Manivel found a general formula for the equation of the tangent variety of the homogeneous spaces which appear in the third line of the magic square. Namely, this formula gives a uniform presentation for the equation of the tangent variety of $v_3(\mathbb{P}^1) \subset \mathbb{P}^3$, $\GG_{w}(3,6) \subset \mathbb{P}^{13}$, $\GG(3,6) \subset \mathbb{P}^{19}$, $\mathbb{S}_{12} \subset \mathbb{P}^{31}$ and $\mathrm{E}_7/\mathrm{P}_7 \subset \mathbb{P}^{55}$.

\end{rem}

\subsection{Four-dimensional linear sections of the quartic $\Q_{\mathrm{Spin}_{12}} \subset \mathbb{P}^{31}$}

In \cite{maniliev}, Iliev and Manivel proved that a generic cubic hypersurface in $\mathbb{P}^8$ can be written, in a finite umber of ways, as a linear section of the $E_6$ invariant cubic  in $\mathbb{P}^{26}$. They observed a similar numerical coincidence in the case of quartic hypersurfaces in $\mathbb{P}^5$. More precisely, let us denote by $\mathcal{M}_{4}^4$ the moduli space of quartic hyperurfaces in $\mathbb{P}^5$. There is a natural map $\Phi : \GG(6,\Delta) /\!\!/ \mathrm{Spin}_{12} \longrightarrow \mathcal{M}_{4}^4$, which is given by restriction of the quartic $\Q_{\mathrm{Spin}_{12}}$ to a given $\mathbb{P}^{5} \subset \mathbb{P}^{31}$. Note that the singular locus of $\sQ$ has dimension $24$, so that a generic $4$-dimensional linear section of $\sQ$ is smooth.

Iliev and Manivel noted that both $\mathcal{M}_{4}^4$ and $\GG(6,\Delta) /\!\!/ \mathrm{Spin}_{12}$ have dimension $90$ and they ask if the map $\Phi$ is dominant. We answer positively to their question :

\begin{theo} \label{section}
Let $\Phi :  \GG(6,\Delta) /\!\!/ \mathrm{Spin}_{12} \longrightarrow \mathcal{M}_{4}^4$ be the map which associates to $L$ the quartic $L \cap \Q_{\mathrm{Spin}_{12}}$. The map $\Phi$ is generically \'etale and dominant. In particular a generic quartic hypersurface in $\mathbb{P}^{5}$ can be represented as a linear section of $\Q_{\mathrm{Spin}_{12}} \subset \mathbb{P}(\Delta)$ in a finite number of ways. As a consequence, the generic quartic double fivefold is a manifold of Calabi-Yau type.
\end{theo}

\begin{proof}

First of all, we notice that $ \GG(6,\Delta) = M_{6 \times \Delta} /\!\!/ \mathrm{GL}_{6}$ and that $\mathcal{M}_{4}^{4} = S^4 \mathbb{C}^{6} /\!\!/  \mathrm{GL}_6$, where $M_{6 \times \Delta}$ is the space of linear maps from $\mathbb{C}^6$ to $\Delta$. Furthermore, the map $\Phi$ is the descent of the natural map:

\begin{equation*}
\begin{split}
\phi \,\,\, : \,\,\, & M_{6 \times \Delta} \longrightarrow S^4 \mathbb{C}^{6} \\
         & (m)_{i,j} \longrightarrow P_{\mathrm{Spin}_{12}}(m_{1,1}.z_1 + \cdots + m_{1,6}.z_6, \cdots, m_{32,1}.z_1 + \cdots + m_{32,6}.z_6),\\ 
\end{split}         
\end{equation*}
where $z_1,\cdots,z_6$ is a basis for $(\mathbb{C}^{6})^*$. In order to prove Theorem \ref{section}, we just need to prove that the differential of $\phi$ is generically surjective. The differential of $\phi$ at a point $(m)_{i,j} \in M_{6 \times \Delta}$ is given by :
 
\begin{equation*}
\begin{split}
d \phi_{(m_{i,j})} \,\,\, : \,\,\, & M_{6 \times \Delta} \longrightarrow S^4 \mathbb{C}^{6} \\
         & (q)_{i,j} \longrightarrow \sum_{i=1}^{32} \left( \sum_{j=1}^{6} q_{i,j}.z_j \right) \times \dfrac{\partial P_{\mathrm{Spin}_{12}}}{\partial x_i}(m_{1,1}.z_1 + \cdots + m_{1,6}.z_6, \cdots, m_{32,1}.z_1 + \cdots + m_{32,6}.z_6),\\
\end{split}
\end{equation*}        
where $x_1, \cdots, x_{32}$ is a basis of $\Delta^*$ in which the equation of $P_{\mathrm{Spin}_{12}}$ is given. As a consequence, in order to prove that $d \phi$ is generically surjective, we only have to prove the lemma below.

\begin{lem}
For generic $m \in M_{6 \times \Delta}$, the dimension of the subspace of $S^{4} \mathbb{C}^{6}$ generated by the:

\begin{equation*}
\{z_j.\dfrac{\partial P_{\mathrm{Spin}_{12}}}{\partial x_i}(m_{1,1}.z_1 + \cdots + m_{1,6}.z_6, \cdots, m_{32,1}.z_1 + \cdots + m_{32,6}.z_6) \}_{i=1 \cdots 32, j=1 \cdots 6}
\end{equation*}
has dimension $126$ and is equal to $S^4 \mathbb{C}^6$.
\end{lem}

\begin{proof}

We use Macaulay$2$ to prove that the lemma. A similar algorithm is provided in \cite{beauville} in order to prove that a general threefold in $\mathbb{P}^{4}$ of degree less than $5$ is Pfaffian. 
\end{proof}
Finally, using proposition $4.5$ (and comments thereafter) of \cite{maniliev}, we deduce from the finite representation statement we just proved that the generic quartic double fivefold is a manifold of Calabi-Yau type.

\end{proof}

\end{section}

\begin{section}{Matrix factorizations on generic quartic double fivefolds}
In this section we prove the existence of a spherical rank $6$ vector bundle on the double cover of $\mathbb{P}^5$ ramified over a general fourfold linear section of $\sQ$. In fact, we will derive the vanishing properties of this vector bundle from those of a rank $3$ coherent sheaf defined on the double quartic fivefold ramified over a general section of $\lQ \subset \mathbb{P}(\bigwedge^3 \mathbb{C}^6)$. Thus, we start our study with a particular matrix factorization on the $\mathrm{SL}_6$-invariant quartic $\lQ \subset \mathbb{P}(\bigwedge^3 \mathbb{C}^6)$, whose cokernel will be the sheaf we are interested in.

\subsection{Foretaste : matrix factorizations on quartic double fivefolds ramified over linear sections of $\lQ \subset \mathbb{P}(\bigwedge^3 \mathbb{C}^6)$}
Let $u_1,\ldots,u_6$ be a basis of $\mathbb{C}^6$ and let $y_1, \ldots, y_{20}$ be coordinates on $\bigwedge^3 \mathbb{C}^6$ such that any $y \in \bigwedge^3 \mathbb{C}^6$ can be written $y = y_1.u_1 \wedge u_2 \wedge u_3 + \ldots + y_{20}.u_5 \wedge u_5 \wedge u_6$. Denote by $P_{\mathrm{SL}_6}$ the equation of $\lQ$. The explicit formula for $\lQ$ is given, for instance, on page $83$ of \cite{kimura}. Denote by:

\begin{equation*}
\begin{split}
Y^{(a)} := \left(
\begin{matrix}
y_{11} & -y_5 & y_2\\
y_{12} & -y_6 & y_3\\
y_{13} & -y_7 & y_4\\
\end{matrix}
\right) & \hspace{2cm}
Y^{(b)} := \left(
\begin{matrix}
y_{10} & -y_9 & y_8\\
y_{16} & -y_{15} & y_{14}\\
y_{19} & -y_{18} & y_{17}\\
\end{matrix}
\right),
\end{split}
\end{equation*}
then we have:

\begin{equation*}
\lP(y) = \left(y_1y_{20} - \operatorname{Tr}(Y^{(a)} Y^{(b)}) \right)^2 + 4y_1 \det(Y^{(b)}) + 4y_{20} \det(Y^{(a)}) - 4 \sum_{i,j}\det(Y_{i,j}^{(a)}) \det(Y_{i,j}^{(b)}),
\end{equation*}
where $Y^{(a)}_{i,j}$ (resp. $Y^{(b)}_{i,j}$) is the matrix obtained from $Y^{(a)}$ (resp. $Y^{(b)}$) by crossing its $i$-th line and $j$-th comlumn.
We recall the expression of a special matrix factorization of $\lP$ that was explicited by Kimura and Sato (see \cite{kimura} page $80$ and $81$). For $k=1 \ldots 6$, define the operators:

\begin{equation*}
\begin{split}
D_k : & \bigwedge^k \mathbb{C}^6 \longrightarrow \bigwedge^{k-1} \mathbb{C}^6 \otimes \mathbb{C}^6\\
& u_{i_1} \wedge \ldots \wedge u_{i_k} \longrightarrow \sum_{r=1}^k (-1)^{k-r} \left(u_{i_1} \wedge \ldots \wedge u_{i_{r-1}} \wedge u_{i_{r+1}} \wedge \ldots \wedge u_{i_k}\right) \otimes u_{i_r}.
\end{split}
\end{equation*}

For each $y \in \bigwedge^{3} \mathbb{C}^6$ and each $z \in \bigwedge^4 \mathbb{C}^6$, we have $(z \otimes 1) \wedge D_3(y) \in \bigwedge^6 \mathbb{C}^6 \otimes \mathbb{C}^6 = \tau \otimes \mathbb{C}^6$, where $\tau = u_1 \wedge \ldots \wedge u_6$ is the canonical volume form on $\bigwedge^6 \mathbb{C}^6$. Hence, there exists a bilinear map $L : \bigwedge^4 \mathbb{C}^6 \times \bigwedge^3 \mathbb{C}^6 \longrightarrow \mathbb{C}^6$ such that $(z \otimes 1) \wedge D_3(y) = \tau \otimes L(z,y)$ for all $y \in \bigwedge^3 \mathbb{C}^6$ and $z \in \bigwedge^4 \mathbb{C}^6$. Now, for each $y \in \bigwedge^3 \mathbb{C}^6$, we define an operator:

\begin{equation*}
\begin{split}
S_{y} : &  \,\,  \mathbb{C}^6 \longrightarrow \mathbb{C}^6\\
& \theta \longrightarrow L(\theta \wedge y,y)
\end{split}
\end{equation*}

Kimura and Sato proves the following (\cite{kimura}, proposition $7$ page $81$):

\begin{prop} \label{kimkim}
For all $y \in \bigwedge^3 \mathbb{C}^6$, we have $S_y^2 = \lP(y). I_6$, where $I_6$ is the $6 \times 6$ identity.
\end{prop}

We will give the explicit form of the matrix $S_y$, when $y = y_1.u_1 \wedge u_2 \wedge u_3 + \ldots + y_{20}.u_5 \wedge u_5 \wedge u_6$. Since it is too big to be reasonably displayed in \LaTeX , we write it in Macaulay$2$ code (and in any case, we will need this presentation to perform computer-aided calculations with it).

\begin{prop} \label{propmatrix}
The matrix of $S_y$ in the basis $u_1, \ldots, u_6$ is :
\begin{verbatim}
S_y = matrix{{y_10*y_11-y_9*y_12+y_8*y_13+y_7*y_14-y_6*y_15+y_5*y_16-y_4*y_17+
y_3*y_18-y_2*y_19+y_1*y_20,
2*(y_13*y_14-y_12*y_15+y_11*y_16),
2*(y_13*y_17-y_12*y_18+y_11*y_19),
2*(y_15*y_17-y_14*y_18+y_11*y_20),
2*(y_16*y_17-y_14*y_19+y_12*y_20),
2*(y_16*y_18-y_15*y_19+y_13*y_20)},

{2*(-y_7*y_8+y_6*y_9-y_5*y_10),
-y_10*y_11+y_9*y_12-y_8*y_13-y_7*y_14+y_6*y_15-y_5*y_16-y_4*y_17+
y_3*y_18-y_2*y_19+y_1*y_20,
2*(-y_7*y_17+y_6*y_18-y_5*y_19),
2*(-y_9*y_17+y_8*y_18-y_5*y_20),
2*(-y_10*y_17+y_8*y_19-y_6*y_20),
2*(-y_10*y_18+y_9*y_19-y_7*y_20)},

{2*(y_4*y_8-y_3*y_9+y_2*y_10),
2*(y_4*y_14-y_3*y_15+y_2*y_16),
-y_10*y_11+y_9*y_12-y_8*y_13+y_7*y_14-y_6*y_15+y_5*y_16+y_4*y_17-
y_3*y_18+y_2*y_19+y_1*y_20,
2*(y_9*y_14-y_8*y_15+y_2*y_20),
2*(y_10*y_14-y_8*y_16+y_3*y_20),
2*(y_10*y_15-y_9*y_16+y_4*y_20)},

{2*(-y_4*y_6+y_3*y_7-y_1*y_10),
2*(-y_4*y_12+y_3*y_13-y_1*y_16),
2*(-y_7*y_12+y_6*y_13-y_1*y_19),
-y_10*y_11-y_9*y_12+y_8*y_13-y_7*y_14+y_6*y_15+y_5*y_16+y_4*y_17-
y_3*y_18-y_2*y_19-y_1*y_20,
2*(-y_10*y_12+y_6*y_16-y_3*y_19),
2*(-y_10*y_13+y_7*y_16-y_4*y_19)},

{2*(y_4*y_5-y_2*y_7+y_1*y_9),
2*(y_4*y_11-y_2*y_13+y_1*y_15),
2*(y_7*y_11-y_5*y_13+y_1*y_18),
2*(y_9*y_11-y_5*y_15+y_2*y_18),
 y_10*y_11+y_9*y_12+y_8*y_13-y_7*y_14-y_6*y_15-y_5*y_16+y_4*y_17+
y_3*y_18+y_2*y_19-y_1*y_20,
2*(y_9*y_13-y_7*y_15+y_4*y_18)},

{2*(-y_3*y_5+y_2*y_6-y_1*y_8),
2*(-y_3*y_11+y_2*y_12-y_1*y_14),
2*(-y_6*y_11+y_5*y_12-y_1*y_17),
2*(-y_8*y_11+y_5*y_14-y_2*y_17),
2*(-y_8*y_12+y_6*y_14-y_3*y_17), 
y_10*y_11-y_9*y_12-y_8*y_13+y_7*y_14+y_6*y_15-y_5*y_16-y_4*y_17-
y_3*y_18+y_2*y_19-y_1*y_20}}
\end{verbatim}
\end{prop}

\begin{proof}
Trivial (though lengthy) handmade computations. One checks with Macaulay$2$ that $S_y^2 = \lP(y) . I_6$.
\end{proof}

\bigskip

As a consequence of proposition \ref{kimkim}, we have a exact sequence in $\mathbb{P}(\bigwedge^3 \mathbb{C}^6)$:

\begin{equation*}
0 \longrightarrow \mathbb{C}^6 \otimes \OO_{\mathbb{P}(\bigwedge^3 \mathbb{C}^6)}(-2) \stackrel{S_y}\longrightarrow \mathbb{C}^6 \otimes \OO_{\mathbb{P}(\bigwedge^3 \mathbb{C}^6)} \longrightarrow \F \longrightarrow 0,
\end{equation*}
where $\F$ is a sheaf scheme-theoretically supported on $\lQ$. Let $L \subset \mathbb{P}(\bigwedge^3 \mathbb{C}^6)$ be a generic $\mathbb{P}^5$. By restricting the above sequence on $L$, we get a sequence:

\begin{equation*}
0 \longrightarrow \mathbb{C}^6 \otimes \OO_{L}(-2) \stackrel{{S_y}|_{L}}\longrightarrow \mathbb{C}^6 \otimes \OO_{L} \longrightarrow \F|_{L} \longrightarrow 0.
\end{equation*}
The matrix $S_L := {S_y}|_{L}$ has full rank at a generic point of $L$ and drops rank on $\lQ \cap L$. We deduce that $\F_L$ is the push-forward of a pure sheaf (say $F_L$) living on $\lQ \cap L$. Since $\mathrm{deg} \left( \lQ \cap L\right) = 4$ and $\mathrm{deg} \left( \det(S_{L}) \right) = 12$, we find that $\mathrm{rank}(F_L) = 3$.

\begin{rem} \label{remG36}
The above argument shows that $\F$ is the push-forward of a rank $3$ sheaf on $\lQ$ (which we denote by $F)$. Remember that $\lQ \subset \mathbb{P}(\bigwedge^3 \mathbb{C}^6)$ is the projective dual of $\GG(3,(\mathbb{C}^6)^*) \subset \mathbb{P}(\bigwedge^3 (\mathbb{C}^6)^*)$ and we have a diagram:

\begin{equation*}
\xymatrix{ & &  \ar[lldd]_{q} I_{\GG(3,(\mathbb{C}^6)^*)/\mathbb{P}(\bigwedge^3 (\mathbb{C}^6)^*)} \ar[rrdd]^{p} & &  \\
& & & & \\
\GG(3,(\mathbb{C}^6)^*) & & & & \lQ}
\end{equation*}
where $I_{\GG(3,(\mathbb{C}^6)^*)/\mathbb{P}(\bigwedge^3 (\mathbb{C}^6)^*)}$ is the projectivization of the conormal bundle of $\GG(3,(\mathbb{C}^6)^*)$ in $\mathbb{P}(\bigwedge^3 (\mathbb{C}^6)^*)$. It is very likely that $F = p_* q^* \mathcal{Q}(m)|_{L}$, where $\mathcal{Q}(m)$ is an appropriate twist of the quotient bundle on $\GG(3,6)$. See sections $3.3$ and $3.4$ of \cite{maniliev}, where a similar phenomenon is shown to be true for the $E_6$-invariant cubic in $\mathbb{P}^{26}$.
 
\end{rem}

If $L \subset \mathbb{P}(\bigwedge^3 \mathbb{C}^6)$ is a generic $\mathbb{P}^5$ with coordinates $z_1,\ldots,z_6$, we denote by $\lP^{(L)}$ the restriction of $\lP$ to $L$. Let $X_{L}$ be the projective subvariety of $\mathbb{P}(1,1,1,1,1,1,2)$ determined by $\lP^{(L)}(z_1,\ldots,z_6) + x^2 = 0$. This is a quartic double fivefold ramified over a $L \cap \lQ$.
We denote by $W^L$ the weighted homogeneous polynomial $\lP^{(L)}(z_1,\ldots,z_6) + x^2$. A result of Orlov (see Theorem $3.11$ and remark $3.12$ of \cite{orlov1}) shows that there is a semi-orthogonal decomposition:

\begin{equation*}
\DB(X_L) = \langle \DB(\mathrm{Gr} W^L) , \OO_{X_L}, \OO_{X_L}(1), \OO_{X_L}(2), \OO_{X_L}(3) \rangle,
\end{equation*}
where $\DB(\mathrm{Gr} W^L)$ is the derived category of graded matrix factorization of the weighted homogeneous polynomial $W^L$. If $X_L$ was smooth then, as explained in \cite{maniliev2} and \cite{kuz1}, the category $\DB(\mathrm{Gr} W^L)$ would be a Calabi-Yau category of dimension $3$. The singular locus of $\lQ \subset \mathbb{P}(\bigwedge^3 \mathbb{C}^6)$ has dimension $14$. Hence, for any linear space $L \subset \mathbb{P}(\bigwedge^3 \mathbb{C}^6)$ of dimension $5$, the variety $X_L$ is singular. Thus, if $E_L$ is a coherent sheaf in $\DB(\mathrm{Gr} W^L)$ whose jumping locus coincide with the singular locus of $X_L$, then one can not expect that $\mathrm{Ext}^3(E_L,E_L) \simeq \mathrm{Hom}(E_L,E_L)^*$ and $\mathrm{Ext}^2(E_L,E_L) \simeq \mathrm{Ext}^1(E_L,E_L)^*$.

\bigskip

It will be however very helpful to keep this idea in mind when we will extend our study to linear sections of $\sQ \subset \mathbb{P}(\Delta)$. For now, we focus on the cohomological properties of the cokernel of a matrix factorization in $\DB(\mathrm{Gr} W^L)$ constructed from $S_L$. Consider the matrices :

\begin{equation*}
\begin{split}
B_L = S_L + ix.I_6 &  \hspace{1.5cm} C_L = S_L - ix.I_6
\end{split}
\end{equation*}
where $S_L$ is the restriction to $L$ of the matrix $S_y$ defined in proposition \ref{propmatrix} and $I_6$ is the $6*6$ identity. We observe that $B_L \times C_L = C_L \times B_L = W^L .I_6$ and that $B_L$ is weighted homogeneous of degree $2$. We deduce that $(B_L,C_L) \in \DB(\mathrm{Gr} W^L)$. The following is the main technical result of this subsection:

\begin{theo} \label{m1} 
Let $L \subset \mathbb{P}(\bigwedge^3 \mathbb{C}^6)$ be a generic $\mathbb{P}^5$ with coordinates $z_1, \ldots, z_6$. Let $B_L = S_L + ix.I_6$ and $C_L = S_L - ix.I_6$, where $S_L$ is the restriction to $L$ of the matrix $S$ explicited in proposition \ref{propmatrix}. Denote by $E_L$ the restriction to $X_L$ of the cokernel of $B_L$. The coherent sheaf $E_L$ has rank $3$ and we have:
\begin{equation*}
\begin{split}
\mathrm{Hom}_{X_L}(E_L,E_L) = \mathbb{C} & \hspace{0.5cm} \textrm{and} \hspace{0.5cm} \mathrm{Ext}_{X_L}^{1}(E_L,E_L) =0.
\end{split}
\end{equation*}
\end{theo}

\begin{proof}
The sheaf $E_L$ has rank $3$ since it is the cokernel of $6*6$ homogeneous matrix with quadratic entries whose degenerating locus is a quartic hypersurface in $\mathbb{P}(1,\ldots,1,2)$. We provide a Macaulay2 code to prove that $\mathrm{Hom}(E_L,E_L) = \mathbb{C}$ and that $\mathrm{Ext}^1(E_L,E_L) = 0$.

\bigskip

We first recall the expression of the matrix $S_y$ we found in proposition \ref{propmatrix} defined over the ring $\mathbb{Z}/313.\mathbb{Z}[i,y_1,\ldots,y_{20},z_1,\ldots,z_6,x]$, where $i$ is the square root of $-1$, $y_1,\ldots, y_{20}$ and $z_1,\ldots,z_6$ have degree $1$ and $x$ has degree $2$.

\begin{verbatim}
i_1 : kk = ZZ/313

i_2 : VV = kk[y_1..y_20]

i_3 : R = kk[w,z_1,z_2,z_3,z_4,z_5,z_6,x,Degrees=>{0,1,1,1,1,1,1,2}]

i_4 : J = ideal(w^2+1)

i_5 : T = R/J

i_6 : U = VV**T

i_7 : S_y = ...
\end{verbatim}
We then create a random $6*20$ matrix with integer coefficients. This matrix represents the equations defining $L$. We then create a matrix (denoted $K1$) which is the restriction of $S_y$ to $L$. For technical reasons while working with Macaulay$2$, we must first define it over $\mathbb{Z}/313.\mathbb{Z}[i,y_1,\ldots, y_{20},z_1,\ldots,z_6,x]$ and then replicate it over $\mathbb{Z}/313.\mathbb{Z}[i,z_1,\ldots,z_6,x]$.
\begin{verbatim}
i_8 : L = random(ZZ^20,ZZ^6,Height=>50)

i_9 : L2 = L**U

i_10 : L3 = L2 * matrix{{z_1},{z_2},{z_3},{z_4},{z_5},{z_6}}

i_11 : K = mutableMatrix(S_y)

i_12 : for k from 0 to 5 do (for i from 0 to 5 do 
(for j from 1 to 20 do K_(k,i) = sub(K_(k,i), y_j => L3_(j-1,0))))

i_13 : K = matrix(K)

i_14 : K1 = mutableMatrix(T,6,6)

i_15 : for k from 0 to 5 do (for j from 0 to 5 do K1_(k,j) = sub(K_(k,j),T))

i_16 : K1 = matrix(K1)
\end{verbatim}

We finally create the matrix $B_L$ and its cokernel $E_L$. We finally compute $\mathrm{Hom}(E_L,E_L)$ and $\mathrm{Ext}^1(E_L,E_L)$.

\begin{verbatim}
i_17 :use T

i_18 : A = matrix{{w*x,0,0,0,0,0},{0,w*x,0,0,0,0},{0,0,w*x,0,0,0},
{0,0,0,w*x,0,0},{0,0,0,0,w*x,0},{0,0,0,0,0,w*x}}

i_19 : B =  K1 + A

i_20 : C = K1-A

i_21 : D = B*C

i_22 : P = D_(0,0)

i_23 : I = ideal(P)

i_24 : T1 = T/I

i_25 : BL = B**T1

i_26 : FL = coker BL

i_27 : X = Proj T1

i_28 : EL = sheaf FL

i_29 : Hom(EL,EL)

i_30 : Ext^1(EL,EL)
\end{verbatim}
In about one hour and a half on a portable workstation, Macaulay$2$ gives the expected answer:

\begin{verbatim}
o_29 : kk^1

o_30 : 0
\end{verbatim}
\end{proof}
\begin{rem}
The vanishing of the $\mathrm{Ext}^1$ for the matrix factorization $B_L$ is somehow quite remarkable. Indeed, one computes with Macaulay$2$ that $\mathrm{Ext}_{\lQ \cap L}^1(F_L,F_L) = \mathbb{C}^{21}$. Hence, going to the double cover of $\mathbb{P}^{5}$ ramified over $\lQ \cap L$ kills the $21$ deformation directions of the cokernel of the initial matrix factorization. We have no explanation for this phenomenon.
\end{rem}

\subsection{Matrix factorizations on quartic double fivefolds ramified over linear sections of $\sQ \subset \mathbb{P}(\Delta)$}

The goal of this subsection is to prove the existence of rank $6$ spherical vector bundles on quartic double fivefolds ramified over general linear sections of $\sQ \subset \mathbb{P}(\Delta)$. Since we proved in the first section of this paper that a general quartic in $\mathbb{P}^5$ is a linear section of $\sQ \subset \mathbb{P}(\Delta)$, this implies that our results hold for a generic quartic double fivefold.

We will use the vanishing result we obtained in section $3.1$ for the rank 3 sheaves $E_L$ which are defined on quartic double fivefolds ramified over general linear sections of $\lQ \subset \mathbb{P}(\bigwedge^3 \mathbb{C}^6)$ and a representation theoretic reduction to go from quartic double fivefolds ramified over linear sections of $\lQ$ to quartic double fivefolds ramified over general linear sections of $\sQ$. We start with a representation theoretic description of the matrix factorization we exhibited in the previous subsection.

\bigskip

Let $\frak{g}$ be a simple Lie algebra over $\mathbb{C}$ and denote by $\beta$ the highest root of $\frak{g}$ (we have chosen a fixed Cartan subalgebra of $\frak{g}$). We say that $\frak{g}$ has a type $\mathrm{II}$ decomposition if there exists a graded decomposition of $\frak{g}$ :

\begin{equation*}
\frak{g} = \frak{g}_{-2} \oplus \frak{g}_{-1} \oplus \frak{g}_{0} \oplus \frak{g}_{1} \oplus \frak{g}_{2},
\end{equation*}
such that $\left[ \frak{g}_i, \frak{g}_k \right] \subset \frak{g}_{i+k}$, $\frak{g}_{-2} = \mathbb{C}.X_{- \beta}$ and $\frak{g}_{2} = \mathbb{C}.X_{\beta}$. If such a decomposition happens, Vinberg \cite{vinberg} proves that $\frak{g}_1$ is a prehomogeneous space under the restricted adjoint action of $G_0$, where $G_0$ is a simply connected group algebraic group with Lie algebra $\frak{g}_0$.

\bigskip

Let $x \in \frak{g}_1$ and let $\mathrm{ad}_x = \left[x,. \right]$ be the adjoint operator associated to $x$. By the grading property, we know that $\ad_x$ is of degree $1$, $(\ad_x)^2$ is of degree $2$ and $(\ad_x)^4$ is of degree $4$. The operator $(\ad_x)^4$ can been then seen as a map:

\begin{equation*}
(\ad_x)^4 : \frak{g}_{-2} \longrightarrow \frak{g}_2.
\end{equation*}
Since $\frak{g}_{-2} = \mathbb{C}.X_{- \beta}$ and $\frak{g}_{2} = \mathbb{C}.X_{\beta}$, we have $(\ad_x)^4(X_{-\beta}) = P(x).X_{\beta}$, where $P$ is a $G_0$-invariant polynomial. We say that $\frak{g}$ has \textit{maximal rank $4$} if the polynomial $P$ is non-identically zero (see \cite{GW1}). In this case, since $\ad_x$ is linear in $x$, we deduce that $P$ must have degree $4$. The polynomial $P$ is thus a degree $4$ relative invariant for the prehomogeneous space $(G_0, \frak{g}_1)$. The maximal rank $4$ representation have been tabulated in \cite{GW1} and they are the following:
\bigskip

\begin{center}
\begin{tabular}{|ccc|ccccc|ccccc|ccccc|}
\hline
& &  & & & & & & & & & & & & & & & \\
 & $\frak{g}$&  & & &$G_0$ & & & & & $\frak{g_1}$ & & & & & $\dim \frak{g_1}$ &&  \\
& &  & & & & & & & & & & & & & & & \\
\hline
& &  & & & & & & & & & & & & & & & \\
 & $B$&  & & &$ \mathrm{SL}_2 \times \mathbb{C}^* \times \mathrm{SO}_d$ & & & & & $\mathbb{C}^2 \otimes \mathbb{C}^d$ & & & & & $2d \,\,(d \,\, \textrm{odd})$ &&  \\
& &  & & & & & & & & & & & & & & & \\
\hline
& &  & & & & & & & & & & & & & & & \\
 & $D$&  & & &$ \mathrm{SL}_2 \times \mathbb{C}^* \times \mathrm{SO}_d$ & & & & & $\mathbb{C}^2 \otimes \mathbb{C}^d$ & & & & & $2d \,\,(d \,\, \textrm{even})$ &&  \\
& &  & & & & & & & & & & & & & & & \\
\hline
& &  & & & & & & & & & & & & & & & \\
 & $F_4$&  & & &$\mathbb{C}^* \times \mathrm{Sp}_6$ & & & & & $(\bigwedge^3 \mathbb{C}^6)_0$ & & & & & $14$ &&  \\
& &  & & & & & & & & & & & & & & & \\
\hline
& &  & & & & & & & & & & & & & & & \\
 & $E_6$&  & & &$\mathbb{C}^* \times \mathrm{SL}_6$ & & & & & $\bigwedge^3 \mathbb{C}^6$ & & & & & $20$ &&  \\
& &  & & & & & & & & & & & & & & & \\
\hline
& &  & & & & & & & & & & & & & & & \\
 & $E_7$&  & & &$\mathbb{C}^* \times \mathrm{Spin}_{12}$ & & & & & $\Delta$ & & & & & $32$ &&  \\
& &  & & & & & & & & & & & & & & & \\
\hline
& &  & & & & & & & & & & & & & & & \\
 & $E_8$&  & & &$\mathbb{C}^* \times E_7$ & & & & & $\textrm{minuscule}$ & & & & & $56$ &&  \\
& &  & & & & & & & & & & & & & & & \\
\hline
\end{tabular}
\end{center}

We focus on the type $\mathrm{II}$ decomposition 
\begin{equation*}
\frak{e}_6 = \mathbb{C} \oplus (\bigwedge^3 \mathbb{C}^6)^* \oplus \frak{gl}_6 \oplus \bigwedge^3 \mathbb{C}^6 \oplus \mathbb{C}.
\end{equation*}

For $y \in \bigwedge^3 \mathbb{C}^6$, the graded decomposition above shows that $(\ad_y)^2 : \frak{gl}_6 \longrightarrow \mathbb{C}$. Hence, we can identify $(\ad_y)^2$ as a $6 \times 6$ matrix. Furthermore, we have the following:

\begin{lem} \label{matfacto}
The pair $((\ad_y)^2, (\ad_y)^2)$ is a matrix factorization of $\lP$.
\end{lem}

\begin{proof}
Consider the $\mathrm{SL}_6$-equivariant map:

\begin{equation*}
\begin{split}
\Pi_{\mathrm{SL}_6} : & \,\,\,\,  \bigwedge^3 \mathbb{C}^6 \longrightarrow \mathrm{End}(\mathbb{C}^6)\\
&\,\,\,\, y \longrightarrow (\ad_y)^2 \circ (\ad_y)^2.
\end{split}
\end{equation*}
Since $(\ad_y)^4 \neq 0$, we know that the map $\Pi$ is not identically zero. The map $\Pi$ is polynomial of degree $4$, so that we get a non-zero $\mathrm{SL}_{6}$-equivariant linear map which lifts $\Pi_{\mathrm{SL}_6}$:

\begin{equation*}
\begin{split}
\widetilde{\Pi_{\mathrm{SL}_6}} : & \,\,\,\,  \mathrm{S}^4(\bigwedge^3 \mathbb{C}^6) \longrightarrow \mathrm{End}(\mathbb{C}^6)\\
\end{split}
\end{equation*}

Furthermore, a computation with $\mathrm{Lie}$ \cite{Lie} shows that the decomposition into $\mathrm{SL}_6$ irreducible representations of $\mathrm{S}^4 (\bigwedge^3 \mathbb{C}^6)$ is :

\begin{equation*}
\mathrm{S}^4 (\bigwedge^3 \mathbb{C}^6) = [0,0,4,0,0] + [1,0,2,0,1] + [2,0,0,0,2] + [0,0,2,0,0] + [0,1,0,1,0] + [0,0,0,0,0],
\end{equation*}
 where $[0,0,0,0,0]$ represents $\lP.\mathbb{C}$. On the other hand, the decomposition of $\mathrm{End}(\mathbb{C}^6)$ into $\mathrm{SL}_6$ irreducible representations is:

\begin{equation*}
\mathrm{End}(\mathbb{C}^6) = [1,0,0,0,1] + [0,0,0,0,0],
\end{equation*}
where $[0,0,0,0,0]$ is $\mathrm{id}_{\mathbb{C}^6}.\mathbb{C}$. By Schur's lemma, we deduce that $\widetilde{\Pi_{\mathrm{SL}_6}} = c.\lP \otimes \mathrm{id}_{\mathbb{C}^6}$, with $c$ a non-zero scalar. Thus, we find that:

\begin{equation*}
(\ad_y)^2 \circ (\ad_y)^2 = c.\lP.\mathrm{id}_{\mathbb{C}^6},
\end{equation*}
with $c \neq 0$. This proves the lemma.
\end{proof}

We will identify this matrix factorization with the one exhibited in the previous section.

\begin{prop} \label{identification}
For all $y \in \bigwedge^3 \mathbb{C}^6$, we have:
$$(\ad_y)^2 = S_y,$$
where $S_y$ is the endomorphism of $\mathbb{C}^6$ defined in proposition \ref{kimkim}
\end{prop}

\begin{proof}
Following \cite{Hitchin}, section 2, one observes that for all $y \in \bigwedge^3 \mathbb{C}^6 \simeq (\bigwedge^3 \mathbb{C}^6)^* $ and for all $\theta \in \mathbb{C}^6$, we have:
\begin{equation*}
\tau \otimes S_y(\theta) = A\left(\iota(\theta,y) \wedge y \right), 
\end{equation*}
where $\iota : \mathbb{C}^6 \times (\bigwedge^3 \mathbb{C}^6)^* \longrightarrow (\bigwedge^2 \mathbb{C}^6)^*$ is the interior product and $A$ is the identification $(\bigwedge^5 \mathbb{C}^6)^* \simeq \tau \otimes \mathbb{C}^6$ given by the volume form $\tau$. Furthermore, following remark $2.17$ of \cite{slupistan}, we notice that $y \in \bigwedge^3 \mathbb{C}^6 $ and for all $\theta \in \mathbb{C}^6$, we have:

\begin{equation*}
(\ad_y)^2 (\theta) = A\left(\iota(\theta,y) \wedge y \right).
\end{equation*}
This concludes the proof of the proposition.
\end{proof}

We are now in position to prove the main result of this section:

\begin{theo} \label{mainmain}
The generic quartic double fivefold is endowed with a spherical rank $6$ vector bundle.
\end{theo}

\begin{proof}
Let $X$ be a generic quartic double fivefold and let $Y$ be the double cover of $\mathbb{P}(\Delta)$ ramified along $\sQ$. By the results of section $2$, we know that $X$ is the fiber product of $Y$ with a generic $L \in \GG(6,\Delta)$.
As a consequence, we have to prove the existence of such a matrix factorization for $X_L = Y \times_{\mathbb{P}(\Delta)} L$, when $L \in \GG(6,\Delta)$ is generic.

\bigskip

Consider the type $\mathrm{II}$ decomposition:

\begin{equation*}
\mathfrak{e_7} = \mathbb{C}^* \oplus \Delta^* \oplus \mathfrak{so}_{12} \oplus \mathbb{C} \oplus \Delta \oplus \mathbb{C}.
\end{equation*}

As $\mathfrak{e_7}$ endowed with this decomposition has maximal rank $4$, we know that:

\begin{equation*}
(\ad_z)^4(X_{-\beta}) = \sP(z).X_{\beta},
\end{equation*}
for all $z \in \Delta$. Since $(\ad_z)^2 \in \mathfrak{so}_{12}^* \simeq \mathfrak{so}_{12} \subset \mathfrak{gl}_{12}$, we can see $(\ad_z)^2$ as a $12*12$ matrix. We have the following :
\begin{lem} \label{matfacto2}
The pair $((\ad_z)^2,(\ad_z)^2)_{z \in \Delta}$ is a matrix factorization of $\sP$. 
\end{lem}

\begin{proof}
The proof is similar to that of Lemma \ref{matfacto}. We check with $\mathrm{Lie}$ \cite{Lie} that the decomposition into $\mathrm{Spin}_{12}$ irreducible representations of $\mathrm{S}^4 \Delta$ is :

\begin{equation*}
\mathrm{S}^4 \Delta = [0,0,0,0,0,4] + [0,1,0,0,0,2] + [0,2,0,0,0,0] + [0,0,0,0,0,2] + [0,0,0,1,0,0] + [0,0,0,0,0,0],
\end{equation*}
 where $[0,0,0,0,0]$ represents $\sP.\mathbb{C}$. Furthermore the decomposition of $\mathrm{End}(\mathbb{C}^{12})$ into $\mathrm{Spin}_{12}$ irreducible representations is:

\begin{equation*}
\mathrm{End}(\mathbb{C}^{12}) = [2,0,0,0,0,0] + [0,1,0,0,0,0] + [0,0,0,0,0,0]
\end{equation*}
where $[0,0,0,0,0,0]$ is $\mathrm{id}_{\mathbb{C}^{12}}.\mathbb{C}$
\end{proof}

The matrix factorization that will be of chief interest for us is closely related to the pair $((\ad_z)^2,(\ad_z)^2)_{z \in \Delta}$. In the following we write $\ad_z^{\mathfrak{g}}$ when we want to specify with which Lie algebra we work.

\bigskip

There is a type $\mathrm{I}$ decomposition:

\begin{equation} \label{equa3}
\mathfrak{so}_{12} = \bigwedge^2 \mathbb{C}^6 \oplus \mathfrak{gl}_6 \oplus \bigwedge^4 \mathbb{C}^6.
\end{equation}

If one chooses the quadratic form on $\mathbb{C}^{12}$ defined on a fixed basis by:

\begin{equation*}
J_{12} = \begin{pmatrix}
 & 0 & 0& \cdots & 0&1 & \\
 & 0 & 0& \cdots & 1&0 & \\
  & \vdots & \vdots& \vdots  & \vdots & \vdots & \\
   & 0 & 1& \cdots & 0&0 & \\
    & 1 & 0& \cdots & 0&0 & \\
\end{pmatrix},
\end{equation*}
then the decomposition of $\mathfrak{so}_{12}$ as in equation (\ref{equa3}) can be restated in the matrix context as follows. Any $P \in \mathfrak{so}_{12}$ can be written:

\begin{equation*}
P = \begin{pmatrix}
 & A & & K & \\
& M & & -A^{t} &\\
\end{pmatrix},
\end{equation*}

where $A \in \mathfrak{gl}_6$, $A^{t}$ is the transpose of $A$ with respect to the anti-diagonal and $K,M$ are skew-symmetric matrices with respect to the anti-diagonal. Hence, for any $z \in \Delta$, we have:

\begin{equation*}
(\ad^{\mathfrak{e_7}}_z)^2 = \begin{pmatrix}
 & A_z & & K_z & \\
& M_z & & -A^{t}_z&\\
\end{pmatrix},
\end{equation*}
where $A_z \in \mathfrak{gl}_6$, $A_z^{t}$ is the transpose of $A_z$ with respect to the anti-diagonal and $K_z,M_z$ are skew-symmetric matrices with respect to the anti-diagonal. Furthermore, we observe that the type $\mathrm{II}$ decomposition:
\begin{equation*}
\mathfrak{e_7} = \mathbb{C}^* \oplus \Delta^* \oplus \mathfrak{so}_{12} \oplus \mathbb{C} \oplus \Delta \oplus \mathbb{C},
\end{equation*}
combined with the type $\mathrm{I}$ decompositions:
\begin{equation*}
\begin{split}
\mathfrak{so}_{12} = \bigwedge^2 \mathbb{C}^6 \oplus \mathfrak{gl}_6 \oplus \bigwedge^4 \mathbb{C}^6 & \hspace{1cm} \Delta = (\mathbb{C}^6)^* \oplus \bigwedge^3 \mathbb{C}^6 \oplus \mathbb{C}^6
\end{split}
\end{equation*}
gives a bigraded decomposition:
\bigskip
\begin{equation*}
\mathfrak{e_7} =  \begin{pmatrix} \mathbb{C}^* &  \oplus  &(\mathbb{C}^6)^* & \oplus & \bigwedge^3 \mathbb{C}^6 & \oplus & \mathbb{C}^6 & \oplus \\
(-2,0) & &(-1,-1) & &(-1,0) & &(-1,1) & \\
& & & & & & & \\
&\bigwedge^2 \mathbb{C}^6 &\oplus& \mathfrak{gl}_6 \oplus \mathbb{C} &\oplus& \bigwedge^4 \mathbb{C}^6 & \oplus & \\
& (0,-1) & &(0,0) & &(0,1) & \\
& & & & & & & \\
 (\mathbb{C}^6)^* & \oplus & \bigwedge^3 \mathbb{C}^6 & \oplus & \mathbb{C}^6 & \oplus & \mathbb{C} \\
(1,-1) & &(1,0) & & (1,1) & &(2,0) \\
\end{pmatrix} 
\end{equation*}
\bigskip

This means that for $z \in \mathfrak{e_7}$ of bidegree $(a,b)$, we have $\ad_z : \mathfrak{e_7}_{(i,k)} \longrightarrow \mathfrak{e_7}_{(i+a,k+b)}$. In particular, if $y \in \bigwedge^3 \mathbb{C}^6$ is of bidegree $(1,0)$, the bigraded decomposition above implies:

\begin{equation*}
(\ad^{\mathfrak{e_7}}_{y})^2 : \bigwedge^2 \mathbb{C}^6 \longrightarrow \{0 \}
\end{equation*}
and
\begin{equation*}
(\ad^{\mathfrak{e_7}}_{y})^2 : \bigwedge^4 \mathbb{C}^6 \longrightarrow \{0 \}.
\end{equation*}
As a consequence, for $y \in \bigwedge^3 \mathbb{C}^6$ of bidegree $(1,0)$, the matrix representation of $(\ad^{\mathfrak{e_7}}_{y})^2$ is:
\begin{equation*}
(\ad^{\mathfrak{e_7}}_{y})^2 = \begin{pmatrix}
 & A_y & & 0 & \\
& 0 & & -A_y^{t} &\\
\end{pmatrix}
\end{equation*}

Furthermore, the above bigraded decomposition of $\mathfrak{e_7}$ can also be obtained (but with swapped bigrading) from the type $\mathrm{I}$ decompositions:
\begin{equation*}
\begin{split}
\frak{e_7} = (V(\omega_6,E_6))^* \oplus \frak{e_6} \oplus \mathbb{C} \oplus V(\omega_6,E_6) & \hspace{1cm}  V(\omega_6,E_6) = (\mathbb{C}^6)^* \oplus \bigwedge^2 \mathbb{C}^6 \oplus \mathbb{C}^6\\
\end{split}
\end{equation*}
and the type $\mathrm{II}$ decomposition:
\begin{equation*}
\frak{e}_6 = \mathbb{C} \oplus (\bigwedge^3 \mathbb{C}^6)^* \oplus \frak{gl}_6 \oplus \bigwedge^3 \mathbb{C}^6 \oplus \mathbb{C}.
\end{equation*}

This establishes that $A_{y}$ is the matrix representation of $(\ad_{y}^{\mathfrak{e}_6})^2$ for any $y \in \bigwedge^{3}\mathbb{C}^6 \subset \Delta$. Hence, thanks to Proposition \ref{identification}, we have for all $y \in \bigwedge^3 \mathbb{C}^6$:

\begin{equation*}
(\ad^{\mathfrak{e_7}}_{y})^2 = \begin{pmatrix}
 & S_y & & 0 & \\
& 0 & & -S_y^{t} &\\
\end{pmatrix}
\end{equation*}
where $S_y$ is the matrix we studied in Section $3.1$.

\bigskip

Let $L$ be a $\mathbb{P}^5$ in $\mathbb{P}(\Delta)$ and denote by $\tilde{B}_L = (\ad^{\mathfrak{e_7}}_{z})^2|_L +ix.I_{12}$ and $\tilde{C}_L = (\ad^{\mathfrak{e_7}}_{z})^2|_L - ix.I_{12}$. We observe that:
\begin{equation*}
\tilde{B}_L \times \tilde{C}_L = \tilde{C}_L \times \tilde{B}_L = (x^2 + \sP^{(L)}(z)).I_{12}.
\end{equation*}
The pair $(\tilde{B}_L,\tilde{C}_L)$ is thus a matrix factorization of the polynomial defining $X_L$ in $\mathbb{P}(1,\ldots,1,2)$. Denote by $\tilde{E}_L$ the cokernel of the restriction to $X_L$ of $\tilde{B}_L$. If $L$ is not included in $\sQ$, then $X_L$ is the degenerating locus of $\tilde{B}_L$. Since $\mathrm{deg}(X_L) = 4$, we find that $\tilde{E}_L$ has generically rank $6$.

\bigskip

Let $L_0$ generic inside $\mathbb{P}(\bigwedge^3 \mathbb{C}^6)$. By the above discussion on the bigraded decomposition of $\mathfrak{e_7}$, we know that:
\begin{equation*}
\tilde{B}_{L_0} = \begin{pmatrix}
 & B_{L_0} & & 0 & \\
& 0 & & -B_{L_0}^{t} &\\
\end{pmatrix}
\end{equation*}
where $B_{L_0}$ is the matrix factorization of $x^2 + \lP^{(L_0)}(z)$ we studied in section $3.1$. As a consequence, $\tilde{E}_{L_0} = E_{L_0} \oplus G_{L_0}$, where $G_{L_0}$ is the restriction to $X_{L_0}$ of the cokernel of $-B_{L_0}^{t}$. One can prove that $G_{L_0} \simeq E_{L_0}^*(2)$, but that won't be useful for the proof. Using a similar Macaulay2 algorithm to the one used for the proof of Theorem \ref{m1}, we find that:

\begin{equation*}
\mathrm{Ext}_{X_{L_0}}^1(E_{L_0}, G_{L_0}) = \mathrm{Ext}_{X_{L_0}}^1(G_{L_0}, E_{L_0}) =  0
\end{equation*}
and 
\begin{equation*}
\mathrm{Hom}_{X_{L_0}}(E_{L_0}, G_{L_0}) = \mathrm{Hom}_{X_{L_0}}(G_{L_0}, E_{L_0}) =  0.
\end{equation*}
By Theorem \ref{m1}, we deduce that:

\begin{equation*}
\mathrm{Ext}_{X_{L_0}}^1(\tilde{E}_{L_0}, \tilde{E}_{L_0}) = 0
\end{equation*}
and

\begin{equation*}
\mathrm{Hom}_{X_{L_0}}(\tilde{E}_{L_0}, \tilde{E}_{L_0}) =  \left\{ 
\begin{pmatrix}
 & \lambda.I_6 & &0 & \\
& 0 & & \mu.I_6 &\\
\end{pmatrix}, \lambda, \mu \in \mathbb{C} \right\}.
\end{equation*}

\bigskip

Let $L \subset \mathbb{P}(\Delta)$ a generic $\mathbb{P}^5$. We will describe the algebra $\mathrm{Ext}^{\bullet}(\tilde{E}_L,\tilde{E}_L)$. Consider $\mathcal{X}$ the family of quartic double fivefolds obtained as a double cover of $\mathbb{P}^5$ ramified along linear sections of $\sQ \subset \mathbb{P}(\Delta)$. We have:

\begin{equation*}
\mathcal{X} = Y \times_{\mathbb{P}(\Delta)} \mathbb{P}(\mathcal{R}),
\end{equation*}
where $\mathbb{P}(\mathcal{R})$ is the projectivization of the tautological bundle over $\GG(6,\Delta)$ and $Y$ is the double cover of $\mathbb{P}(\Delta)$ ramified along $\sQ$. The natural projection $\pi : \mathcal{X} \longrightarrow \GG(6, \Delta)$ is a proper and flat morphism. Let $\tilde{E}$ be the rank $6$ sheaf defined on $Y$ as the cokernel of the restriction to $Y$ of the matrix representing $(\ad^{\mathfrak{e_7}}_{z})^2 + x.I_{12}$, for $z \in \mathbb{P}(\Delta)$ and $x$ of degree $2$. We denote by $\tilde{\E}$ the pull-back of $\tilde{E}$ to $\X$. We observe that for any $L \in \GG(6, \Delta)$ which intersects $\sQ$ properly, we have:

\begin{equation*}
\tilde{\E}|_{\pi^{-1}(L)} =  \tilde{E}_L.
\end{equation*}

By the discussion above, we know that:

\begin{equation*}
\mathrm{Ext}_{X_{L_0}}^1(\tilde{E}_{L_0}, \tilde{E}_{L_0}) = 0
\end{equation*}
and

\begin{equation*}
\mathrm{Hom}_{X_{L_0}}(\tilde{E}_{L_0}, \tilde{E}_{L_0}) =  \left\{ 
\begin{pmatrix}
 & \lambda.I_6 & &0 & \\
& 0 & & \mu.I_6 &\\
\end{pmatrix}, \lambda, \mu \in \mathbb{C} \right\},
\end{equation*}
for generic $L_0 \in \GG(6,\bigwedge^3 \mathbb{C}^6)$. Since the morphism $\pi : \X \longrightarrow \GG(6,\Delta)$ is proper and flat and the base $\GG(6,\Delta)$ is smooth, the semi-continuity Theorem implies that:

\begin{equation*}
\mathrm{Ext}_{X_{L}}^1(\tilde{E}_{L}, \tilde{E}_{L}) = 0
\end{equation*}
and
\begin{equation*}
\mathrm{Hom}_{X_{L}}(\tilde{E}_{L}, \tilde{E}_{L}) \subset \left\{ 
\begin{pmatrix}
 & \lambda.I_6 & &0 & \\
& 0 & & \mu.I_6 &\\
\end{pmatrix}, \lambda, \mu \in \mathbb{C} \right\}.
\end{equation*}

Furthermore, for generic $L$, we have:
\begin{equation*}
\tilde{B}_{L} = \begin{pmatrix}
 & A_L & & K_L & \\
& M_L & & -A^{t}_L &\\
\end{pmatrix},
\end{equation*}
with $K_L,M_L \neq 0$. We deduce that $\mathrm{Hom}_{X_{L}}(\tilde{E}_{L}, \tilde{E}_{L}) = \mathbb{C}.I_{12}$. Finally, for generic $L \in \GG(6,\Delta)$, the section $L \cap \sQ$ is smooth and by \cite{kuz1}, we have a decomposition:

\begin{equation*}
\DB(X_L) = \langle \A_{X_L}, \OO_{X_L}, \OO_{X_L}(1), \OO_{X_L}(2), \OO_{X_L}(3) \rangle,
\end{equation*}
where $\A_{X_L}$ is a $CY$-3 category. Using the resolution:

\begin{equation*}
0 \longrightarrow \mathbb{C}^{12} \otimes \OO_{\mathbb{P}(1,\ldots,1,2)}(-2) \longrightarrow \mathbb{C}^{12} \otimes \OO_{\mathbb{P}(1,\ldots,1,2)} \longrightarrow i_*(\tilde{E}_L) \longrightarrow 0,
\end{equation*}
we easily prove that $\tilde{E}_L(-1) \in \A_{X_L}$ and $\tilde{E}_L(-2) \in \A_{X_L}$. Since $\A_{X_L}$ is a CY-$3$ category, we deduce that $\mathrm{Ext}^3(\tilde{E}_L,\tilde{E}_L) = \mathrm{Hom}(\tilde{E}_L,\tilde{E}_L)^* = \mathbb{C}$ and that $\mathrm{Ext}^{2}(\tilde{E}_L,\tilde{E}_L) = \mathrm{Ext}^1(\tilde{E}_L,\tilde{E}_L) = 0$. This establishes that $\tilde{E_L}(-1)$ and $\tilde{E}_{L}(-2)$ are spherical rank $6$ vector bundles contained in $\A_{X_L}$ and finishes the proof of Theorem \ref{mainmain}.

\end{proof}

The following is analogous to remark \ref{remG36}.
\begin{rem}
It is certainly worth noting that the cokernel of the matrix factorization $\left((\ad^{\mathfrak{e_7}}_z)^2,(\ad^{\mathfrak{e_7}}_z)^2 \right)$ is a rank $6$ coherent sheaf living on $\sQ \in \mathbb{P}(\Delta)$. We have a diagram:

 \begin{equation*}
\xymatrix{ & &  \ar[lldd]_{q} I_{\mathbb{S}_ {12^*}/\mathbb{P}(\Delta^*)} \ar[rrdd]^{p} & &  \\
& & & & \\
\mathbb{S}_{12^*} & & & & \sQ}
\end{equation*}
where $\mathbb{S}_{12^*} \subset \mathbb{P}(\Delta^*)$ is one of the connected component of the orthogonal Grassmannian $\mathrm{OG}(6,12^*)$ in the (dual) spinor embedding and $I_{\mathbb{S}_ {12^*}/\mathbb{P}(\Delta^*)}$ is the projectivization of the conormal bundle of $\mathbb{S}_{12}$ in $\mathbb{P}(\Delta^*)$. It is very likely that this rank $6$ coherent sheaf is $p_* q^* \mathcal{Q}(m)|_{L}$, where $\mathcal{Q}(m)$ is an appropriate twist of the quotient bundle on $\mathbb{S}_{12^*}$.

\end{rem}

\subsection{Homological unit of the CY-$3$ category associated to a generic quartic double fivefold}

In \cite{homounit}, the concept of \textit{homological unit} was introduced for a large class of triangulated categories as a replacement for the algebra $H^{\bullet}(\OO_X)$ when the category under study is not (necessarily) the derived category of a projective variety. This notion has been further explored in \cite{krug}, where intriguing examples of units have been constructed and in \cite{HKC}, where it was used to define hyper-K\"ahler categories. We give a quick reminder on the definition and basic properties of the homological units. 

\begin{defi}
Let $\C$ be an abelian category and $\varphi : \C \longrightarrow \mathbb{N}$ be a function. We say that $\varphi$ is a \textbf{rank function} if it is additive with respect to exact sequences in $\C$. We say that $\varphi$ is trivial if its image is $\{ 0 \}$.
\end{defi}

We provide the definition of \textit{homological units} for an abelian category endowed with a rank function.

\begin{defi} \label{homounit}
Let $\C$ be an abelian category with enough injectives and with a non-trivial rank function. Let $\T$ be a full admissible subcategory in $\DB(\C)$. A graded algebra $\mathfrak{T}^{\bullet}$ is called a \textbf{homological unit} for $\T$ (with respect to $\C$), if $\mathfrak{T}^{\bullet}$ is maximal for the following properties : 
\begin{enumerate}

\item for any object $\F \in \T$, there exists a pair of morphisms $i_{\F} : \mathfrak{T}^{\bullet} \rightarrow  \HHH^{\bullet}(\F,\F)$ and $t_{\F} : \HHH^{\bullet}(\F, \F) \rightarrow \mathfrak{T}^{\bullet}$ with the properties:
 \begin{itemize}

 \item the morphism $i_{\F} : \mathfrak{T}^{\bullet} \rightarrow  \HHH^{\bullet}(\F,\F)$ is a graded algebra morphism which is functorial in the following sense. Let $\F, \G \in \T$ and let $a \in \mathfrak{T}^{k}$ for some $k$. Then, for any morphism $\psi : \F \rightarrow \G$, there is a commutative diagram:
\begin{equation*}
\xymatrix{ \F \ar[rr]^{i_{\F}} \ar[dd]^{\psi} & &\F [k] \ar[dd]^{\psi[k]} \\
& &  \\
\G \ar[rr]^{i_{\G}} & & \G [k]}
\end{equation*}

\item the morphism $t_{\F} : : \HHH^{\bullet}(\F, \F) \rightarrow \mathfrak{T}^{\bullet}$ is a graded vector spaces morphism which satisfies the dual functoriality property of $i_{\F} $.
\end{itemize}

\item for any $\F \in \T$ which rank (seen as an object in $\DB(\C)$) is not vanishing, the morphism $t_{\F}$ splits $i_{\F}$ as a morphism of graded vector spaces.

\end{enumerate}

With hypotheses as above, an object $\F \in \T$ is said to be \textup{unitary}, if $\HHH^{\bullet}(\F,\F) = \mathfrak{T}^{\bullet}$, where $\mathfrak{T}^{\bullet}$ is a homological unit for $\T$.

\end{defi}

We notice the following facts (see \cite{homounit}):
\begin{enumerate}
\item If $\T$ contains a unitary object whose rank is not zero with respect to the chosen rank function on $\DB(\C)$, then the homological unit with respect to this rank function is necessarily unique. This follows from the maximality condition imposed in definition \ref{homounit}.
\item Let $\T$ be an admissible subcategory of $\DB(Coh(X)^{\mathrm{G}})$ where $X$ is a smooth projective variety and $\mathrm{G}$ a reductive group acting linearly on $X$. Assume that there exists a $\mathrm{G}$-equivariant line bundle $L \in \mathrm{Pic}(X)$ such that $L \in \T$. Then, the homological unit of $\T$ with respect to the natural rank function is $H^{\bullet}(\OO_X)^{\mathrm{G}}$.
\item Let $X$ and $Y$ be smooth projective varieties of dimension less or equal to $4$ such that $\DB(X) \simeq \DB(Y)$. It is proved in \cite{homounit} that the algebras $H^{\bullet}(\OO_X)$ and $H^{\bullet}(\OO_Y)$ are isomorphic. This suggests that the homological unit of a triangulated category of geometric origin could be independent of the embedding into the derived category of a smooth projective variety (at least if the dimensions of the varieties are small enough).

\end{enumerate}

In case $\T$ contains a spherical object whose rank is non-zero, the homological unit is easily computed:

\begin{lem}
Let $\T \subset \DB(X)$ be an admissible subcategory of the derived category of a smooth projective variety. Assume that $\T$ is a CY-$n$ category and that it contains a spherical object whose rank is non-zero. Then, the homological unit of $\T$ (with respect to the natural rank function coming from $\DB(X)$) is $\mathbb{C} \oplus \mathbb{C}[n]$.
\end{lem}

\begin{proof}
Since $\T$ is a CY-$n$ category, Serre duality shows the algebra $\mathbb{C} \oplus \mathbb{C}[n]$ embeds functorially via the trace map into the algebra $\mathrm{Ext}^{\bullet}(\F,\F)$, for any $\F \in \T$ whose rank is non-zero. Furthermore, the category $\T$ contains an object whose rank is non-zero and whose Ext-algebra is $\mathbb{C} \oplus \mathbb{C}[n]$. We deduce that the homological unit of $\T$ is $\mathbb{C} \oplus \mathbb{C}[n]$.
\end{proof}

From this lemma and Theorem \ref{mainmain}, we deduce the:

\begin{cor}
Let $X$ be a generic quartic double fivefold and let $\A_{X}$ the CY-$3$ category associated to the derived category of this fivefold. The homological unit of $\A_X$ is $\mathbb{C} \oplus \mathbb{C}[3]$.
\end{cor}

This contrasts sharply with the fact that $H^{\bullet}(\OO_X) = \mathbb{C}$, when $X$ is a quartic double fivefold. This corollary also suggests a natural question on the homological units of the CY-$3$ categories potentially associated to a manifolds of Calabi-Yau type, namely:

\begin{quest}
Let $X$ be a manifold of Calabi-Yau type. Assume that a semi-orthogonal component of the derived category of $X$ is a CY-$3$ category. Is the homological unit of this CY-$3$ category $\mathbb{C} \oplus \mathbb{C}[3]$?
\end{quest}

It is proved in \cite{maniliev}, that the CY-$3$ category associated to the derived category of a generic cubic sevenfold contains a spherical rank $9$ vector bundle. Hence, by the above lemma, the homological unit of this CY-$3$ category is $\mathbb{C} \oplus \mathbb{C}[3]$. Thus, as far as manifolds of Calabi-Yau type obtained as generic complete intersections of dimension bigger than $4$ in weighted projective spaces are concerned, in order to answer the above question we only need to prove that the CY-$3$ category associated to the transverse intersection of a generic cubic and a generic quadric in $\mathbb{P}^7$ contains a spherical vector bundle.

\end{section}

\newpage

\bibliographystyle{alpha}

\bibliography{dqf}
\newpage

\appendix



\end{document}